\documentclass[a4paper,11pt,oneside,reqno]{amsart} 
\usepackage{amsfonts,amssymb,amsmath,amsthm}
\usepackage{graphicx}
\usepackage{hyperref}
\usepackage{pdfsync}
\usepackage{bbm}
\newtheorem{theorem}{Theorem}[section]
\newtheorem{lemma}[theorem]{Lemma}

\newtheorem{corollary}[theorem]{Corollary}

\theoremstyle{definition}
\newtheorem{definition}[theorem]{Definition}

\theoremstyle{remark}
\newtheorem{remark}{Remark}
\newtheorem{example}{Example}

\numberwithin{equation}{section}

\def\Var{\mathop{\rm Var}\nolimits}

\def\<{\langle}
\def\>{\rangle}

\newcommand{\E}{\ensuremath{\mathbb{E}}}

\begin{document}

\title[Eigenvalues of random permutation matrices]{On fluctuations of eigenvalues of random permutation matrices}

\author[G. Ben Arous]{G\'{e}rard Ben Arous}
\address{G. Ben Arous\\
 Courant Institute of the Mathematical Sciences\\
 New York University\\
 251 Mercer Street\\
 New York, NY 10012, USA}
\email{benarous@cims.nyu.edu}
\author[K. Dang]{Kim Dang}
\address{K. Dang\\
Institut f\"{u}r Mathematik\\
Universit\"{a}t Z\"{u}rich\\
Winterthurerstrasse 190\\
CH-8057 Z\"{u}rich, Switzerland}
\email{kim.dang@math.uzh.ch}


\date{\today}

\maketitle
\begin{abstract}
Smooth linear statistics of random permutation matrices, sampled under a general Ewens distribution, exhibit an interesting non-universality phenomenon.
Though they have bounded variance, their fluctuations are asymptotically non-Gaussian but infinitely divisible. The fluctuations are asymptotically Gaussian for less smooth linear statistics for which the variance diverges. The degree of smoothness is measured in terms of the quality of the trapezoidal approximations of the integral of the observable.
\end{abstract}
\section{Introduction}
We study the fluctuations of the spectrum of random permutation matrices, or more precisely, of their linear statistics under the Ewens distribution for a wide class of functions. 
The study of linear statistics of the spectrum of random matrices is an active field (for results concerning invariant ensembles, see for instance \cite{costin_lebowitz}, \cite{diaconisevans}, \cite{diaconissh}, \cite{dumitriu_edelman}, \cite{johansson1}, \cite{johansson}, \cite{johnson}, \cite{pastur} or \cite{wieand2} and for non-invariant ensembles see for instance \cite{bai}, \cite{baisilverstein}, \cite{boutet_pastur_shcherbina}, \cite{chatterjee}, \cite{dumitriu_soumik}, \cite{khorunzhy}, \cite{lytova} or \cite{wieand}). All previous results (except \cite{pastur}) have two common features. Firstly,  the variance of linear statistics does not diverge for smooth enough functions, and thus, no normalization is needed to get a limit law, whereas for less smooth functions the variance blows up very slowly (i.e. logarithmically). The second feature is that those fluctuations are asymptotically Gaussian (except in \cite{pastur} again, where invariant ensembles with more than one cut are shown to have non-Gaussian fluctuations).  We will see that the behavior of the variance of the  linear statistics of random permutation matrices  follow the general pattern. But we will also prove that the asymptotic limit law is more surprising, in that it is not Gaussian but infinitely divisible when the function is smooth enough. This is in contrast to the case where the function is less regular, the fluctuations being then indeed asymptotically Gaussian. This Gaussian behavior was previously proved by K. Wieand (see \cite{wieand}) for the special case where the linear statistic is the number of eigenvalues in a given arc, and  for uniformly distributed permutations.\newline

We first introduce our notations. If $N$ is an integer,  $\mathcal{S}_N$ will denote the symmetric group.
We denote by  $M_\sigma$ the permutation matrix defined by the permutation $\sigma\in\mathcal{S}_N$.
For any $1 \leq i,j \leq N$, the entry $M_\sigma (i,j)$  is given by
\begin{equation}
M_\sigma (i,j)=\mathbbm{1}_{i=\sigma(j)} .
\end{equation}

$M_\sigma$ is unitary, its eigenvalues belong to the unit circle $\mathbb{T}$. We denote them  by 
\begin{equation}
\lambda_1(\sigma)=e^{2i\pi\varphi_1(\sigma)},\dots,\lambda_N(\sigma)=e^{2i\pi\varphi_N(\sigma)}\in\mathbb{T},
\end{equation}

where $\varphi_1(\sigma),\dots,\varphi_N(\sigma)$ are in $[0,1]$.

For any real-valued periodic function $f$ of period $1$, we define the linear statistic

\begin{equation}\label{sum1}
I_{\sigma,N}(f):=\textrm{Tr}\widetilde{f}(M_\sigma)= \sum_{i=1}^Nf(\varphi_i(\sigma)),
\end{equation}
where $\widetilde{f}(e^{2i\pi\varphi})=f(\varphi)$ is a function on the unit circle $\mathbb{T}$.

We consider random permutation matrices by sampling $\sigma$ under the Ewens distribution 
\begin{equation}
\nu_{N,\theta}(\sigma)=\frac{\theta^{K(\sigma)}}{\theta(\theta+1)\dots(\theta+N-1)},
\end{equation}
where $\theta>0$ and $K(\sigma)$ is the total number of cycles of the permutation $\sigma$.
The case $\theta=1$ corresponds to the uniform measure on $\mathcal{S}_N$. \\

We study here the asymptotic behavior of the linear statistic $I_{\sigma, N}(f)$ under the Ewens distribution $\nu_{N,\theta}$ for any $\theta>0$, and a wide class of functions $f$. 
As mentioned above, the asymptotic behavior depends strongly on the smoothness of $f$. In order to quantify this dependence, we introduce the sequence
\begin{equation}\label{Rj}
R_j(f)=\frac 1j \sum_{k=0}^{j-1}f\left(\frac kj\right)-\int_0^1f(x)dx.
\end{equation}
Using the periodicity of $f$, it is clear that 
\begin{equation}\label{Rj1}
R_j(f)=\frac 1j\left(\frac 12 f(0)+\sum_{k=1}^{j-1}f\left(\frac kj\right)+\frac 12 f(1)\right)-\int_0^1f(x)dx.
\end{equation}
So that $R_j(f)$ is easily seen to be the error in the composite trapezoidal approximation to the integral of $f$ \cite{davis}. 

We will see that the asymptotic behavior of the linear statistic $I_{\sigma, N}(f)$ is controlled  by the asymptotic behavior of the $R_j(f)$'s, when $j$ tends to infinity, i.e. by the quality of the composite trapezoidal approximation to the integral of $f$. The role played by the quality of the trapezoidal approximation of $f$ might seem surprising, but it is in fact very natural. It is a simple consequence of the fact that the spectrum of the permutation matrix $M_\sigma$ is easily expressed in terms of the cycle counts of the random permutation  $\sigma$, i.e. the numbers $\alpha_j(\sigma)$ of cycles of length $j$, for $1\leq j\leq N$. 
Indeed, the spectrum of $M_\sigma$ consists in the union, for $1\leq j \leq N$, of the sets of  $j$-th roots of unity, each taken with multiplicity $\alpha_j(\sigma)$. 
This gives 

\begin{equation}
I_{\sigma,N}(f)=\sum_{j=1}^N\alpha_j(\sigma)\sum_{\omega^j=1}\widetilde{f}(\omega)= \sum_{j=1}^N\alpha_j(\sigma)\sum_{k=0}^{j-1}f\left(\frac kj\right).
\end{equation}

So that, using the definition \eqref{Rj} of the $R_j$'s, and the obvious fact that $\sum_{j=1}^N j\alpha_j(\sigma)=N$, it becomes clear that:

\begin{equation}\label{RandI}
I_{\sigma,N}(f)= N\int_0^1f(x)dx+\sum_{j=1}^N\alpha_j(\sigma)jR_j(f).
\end{equation}

At this point, and using the basic equality (\ref{RandI}), it is easy to explain intuitively the non-universality phenomenon we have uncovered in this work. When the function $f$ is smooth enough, the sequence $R_j(f)$  converges fast enough to zero to ensure that the linear statistic is well approximated by the first terms in the sum \eqref{RandI}. These terms correspond to the well separated eigenvalues associated with small cycles. The discrete effects related to these small cycles in the spectrum are then dominant, and are responsible for the non-Gaussian behavior. Thus, the appearance of non-universal fluctuations is due to a very drastic localization phenomenon. Indeed, the important eigenvalues for the behavior of smooth linear statistics are atypical in the sense that they correspond to very localized eigenvectors, those localized on small cycles. When the function is less smooth, the variance will diverge (slowly) so that a normalization will be necessary. After this normalization, the discrete effects will be washed away and the limit law will be Gaussian.\newline

We will first describe the fluctuations of linear statistics of smooth enough functions $f$, i.e. in the case when the $R_j(f)$'s decay to $0$ fast enough to ensure that the variance of the linear statistic stays bounded.
 
\begin{theorem}\label{Thm}
Let  $\theta$ be any positive number, and $f$ be a function of bounded variation.
Assume that 
\begin{equation}\label{Hyp}
\sum_{j=1}^{\infty}jR_j(f)^2 \in (0,\infty).
\end{equation}
 Then, 
\begin{enumerate}

\item  under the Ewens distribution $\nu_{N,\theta}$, the distribution of the centered linear statistic 
$$I_{\sigma,N}(f)-\mathbb{E}[I_{\sigma,N}(f)]$$
 converges weakly, as $N$ goes to infinity, to a non-Gaussian infinitely divisible distribution $\mu_{f,\theta}$.
 
 \item The  distribution $\mu_{f,\theta}$ is defined by its Fourier transform 

\begin{equation}\label{mufourier}\widehat{\mu}_{f,\theta}(t)=\exp\left(\theta\int(e^{itx}-1-itx)dM_f(x)\right),\end{equation}
where the L\'{e}vy measure $M_f$ is given by
 \begin{equation}\label{measuremu}
       M_f=\sum_{j=1}^\infty\frac 1j\delta_{jR_j(f)}.
  \end{equation} 
  
  \item The  asymptotic behavior of the expectation of the linear statistic is given by
  \begin{equation}
\E[{I_{\sigma, N}}(f)]= N\int_0^1 f(x)dx+\theta \sum_{j=1}^NR_j(f) + o(1).
\end{equation}
Here, the second term $\sum_{j=1}^NR_j(f)$ may diverge, but not faster than logarithmically.
\begin{equation}
\sum_{j=1}^NR_j(f) = O(\sqrt{\log N})
\end{equation}
  
  \item The asymptotic behavior of the variance of the linear statistic is given by
  \begin{equation}
\Var[{I_{\sigma, N}}(f)]= \theta \sum_{j=1}^N jR_j(f)^2 +o(1).
\end{equation}

  \end{enumerate}
  
\end{theorem}

\begin{remark}
In this theorem (and in the next), we restrict ourselves to the class of functions $f$ of bounded variation. This is not at all a necessary hypothesis, but it simplifies greatly the statements of the theorems. 
Our proofs give more. We will come back later (in Section 2) to the best possible assumptions really needed for each statement. These assumptions involve the notion of Cesaro means of fractional order, which we wanted to avoid in this introduction.
\end{remark}

\begin{remark}
We  note that the assumption \eqref{Hyp} is not satisfied in the trivial case where $f$ is in the kernel of the composite trapezoidal rule, i.e when the composite trapezoidal rule gives the exact approximation to the integral of $f$ for all $j$'s. In this case, the sequence $R_j(f)$ is identically zero and the linear statistic is non-random. 
Obviously, this is the case for every constant function $f$ and for every odd function $f$, i.e. if 
\begin{equation}
f(x)=-f(1-x).
\end{equation}
It is indeed easy to see then that $R_j(f)=0$ for all $j\geq1$. 
\end{remark}

\begin{remark}
Consider now the even part of $f$, i.e.
\begin{equation}
f_{\textrm{even}}(x)=\frac 12\left(f(x)+f(1-x)\right).
\end{equation} 
It is clear then that 
\begin{equation}R_j(f)=R_j(f_{\textrm{even}}),
\end{equation} 
so that the assumption \eqref{Hyp} in fact only deals with the even part of $f$.
\end{remark}

\begin{remark} 
In order to avoid the possibility mentioned above for all $R_j(f)$'s to be zero, we introduce the following assumption 
 \begin{equation}\label{A6}
f_{\textrm{even}}\textrm{ is not a constant.}
\end{equation}
Note that, in general, it is not true that \eqref{A6} implies that the sequence of $R_j(f)$'s is not identically  zero, even when $f$ is continuous! (See \cite{hille} or \cite{loxtonsanders}.)
But when $f$ is in the Wiener algebra, i.e. when its Fourier series converges absolutely, then \eqref{A6} does imply that one of the $R_j(f)$'s is non zero  (see \cite{loxtonsanders}, p. 260). 
\end{remark}

\begin{remark}\label{expandvar}
It is in fact easy to compute explicitly the value of the expectation and variance of the linear statistic $I_{\sigma, N}(f)$ for any value of $\theta$ and of $N$.
This is done below, in Section~\ref{expectation_and_variance}. The asymptotic analysis is not immediate for the values of $\theta<1$.
\end{remark}

We now want to show how the assumption \eqref{Hyp} can easily be translated purely in terms of manageable regularity assumptions on the function $f$ itself.

\begin{corollary}\label{cor1}
If $f\in\mathcal{C}^1$, let $\omega(f',\delta)$ be the modulus of continuity of its derivative $f'$. 
Assume that 
\begin{equation}\label{omegacond}
\sum_{j=1}^\infty\frac1j\omega(f',1/j)^2<\infty,
\end{equation}
also assume \eqref{A6} in order to avoid the trivial case mentioned above, 
then the conclusions of Theorem~\ref{Thm} hold.
\end{corollary}

Of course, the condition \eqref{omegacond} is satisfied if $f\in\mathcal{C}^{1+\alpha}$, for $0<\alpha<1$, i.e. if $f'$ is $\alpha$-H\"{o}lder continuous. 

We can give variants of the assumptions of smoothness of $f$ given in Corollary~\ref{cor1}. For instance,

\begin{corollary}\label{cor2}
If $f$ has a derivative in $L^p$, let  $\omega^{(p)}(f',\delta)$ be the modulus of continuity in $L^p$ of its derivative $f'$, i.e.
\begin{equation}
\omega^{(p)}(f',\delta)=\sup_{0\leq h\leq \delta}\left\{\int_0^1|f'(x+h)-f'(x)|^p\right\}^{1/p}.
\end{equation}
Assume that
\begin{equation}\label{cor2omega}
\omega^{(p)}(f',\delta)\leq\delta^\alpha\quad\textrm{with }\alpha>\frac 1p,
\end{equation}
also assume \eqref{A6} in order to avoid the trivial case mentioned above, 
then the conclusions of Theorem~\ref{Thm} hold.
\end{corollary}

It is of course also possible to relate the $R_j(f)$'s to the Fourier coefficients of $f$. Indeed, if the Fourier series of $f$ 
\begin{equation}
f(x)=a_0+\sum_{n=1}^\infty a_n\cos(n2\pi i x)+\sum_{n=1}^\infty b_n\sin(n2\pi i x)
\end{equation}
converges, then the Poisson summation formula shows that
\begin{equation}\label{poissonsummation}
R_j(f)=\sum_{n=1}^\infty a_{jn}.
\end{equation}

Using this relation, it is easy to prove the following Corollary:
\begin{corollary}\label{cor3}
 If $f$ is in the Sobolev space $H^{s}$, for $s>1$, and if one assumes \eqref{A6}, then the conclusions of Theorem~\ref{Thm} hold.
\end{corollary}

\begin{remark}
The formula \ref{poissonsummation} gives an expression for the asymptotic variance of the linear statistic
\begin{equation}
\lim_{N \to \infty}\Var[{I_{\sigma, N}}(f)]= \theta \sum_{j=1}^{\infty} jR_j(f)^2 = \theta\sum_{k,l=1}^{\infty}a_ka_ld(k,l),
\end{equation}
where  $d(k,l)$ is the sum of the divisors of the integers $k$ and $l$.
\end{remark}

We now give two interesting examples of functions satisfying the conditions of Theorem~\ref{Thm}:

\begin{example}
Let $f$ be a trigonometric polynomial of degree $k$. Then, $R_j(f)=0$ for all $j>k$. Obviously, the condition \eqref{Hyp} of Theorem~\ref{Thm} is satisfied and the limit distribution $\mu_{f,\theta}$ is a compound Poisson distribution with $M_f$ given by
\begin{equation}
M_f=\theta\sum_{j=1}^k\frac 1j\delta_{jR_j}.
\end{equation}
\end{example}

\begin{example}
Let $f\in\mathcal{C}^\infty$ and $f\equiv 1$ on $[a,b]$ and $f\equiv 0$ on $[a-\epsilon, b+\epsilon]^c$, then the result of Theorem~\ref{Thm} applies. So, the centered linear statistic $I_{\sigma_N}(f)-\mathbb{E}[I_{\sigma_N}(f)]$ has a finite variance and a non-Gaussian infinitely divisible limit distribution. This is a very different behavior from the case $f=\mathbbm{1}_{[a,b]}$ (see below), where the limit is Gaussian.
\end{example}

We now give our second main result, i.e. sufficient conditions ensuring that the variance of the linear statistic $I_{\sigma,N}(f)$ diverges and that the linear statistic converges in distribution to a Gaussian, when centered and normalized.

\begin{theorem}\label{Thm2}
Let  $\theta$ be any positive number, and $f$ be a bounded variation function such that 
\begin{equation} \label{Hyp2}
\sum_{j=1}^{\infty}jR_j(f)^2 = \infty.
\end{equation}
 Then, 
\begin{enumerate}

\item  under the Ewens distribution $\nu_{N,\theta}$, the distribution of the centered and normalized linear statistic 
\begin{equation}
\frac{I_{\sigma,N}(f)-\mathbb{E}[I_{\sigma,N}(f)]}{\sqrt{\Var{I_{\sigma, N}(f)}}}
\end{equation}
 converges weakly, as $N$ goes to infinity, to the Gaussian standard distribution $\mathcal{N}(0,1)$. 
  
  \item The  asymptotic behavior of the expectation of the linear statistic is given by
  \begin{equation}
\E[{I_{\sigma, N}}(f)]= N\int_0^1 f(x)dx+ \theta \sum_{j=1}^NR_j(f) + O(1).
\end{equation}
Here, the second term $\sum_{j=1}^NR_j(f)$ may diverge, but not faster than logarithmically.
\begin{equation}
\sum_{j=1}^NR_j(f) = O(\log N)
\end{equation}
  
  \item The asymptotic behavior of the variance of the linear statistic is given by
  \begin{equation}
\Var[{I_{\sigma, N}}(f)] \sim \theta \sum_{j=1}^N jR_j(f)^2. 
\end{equation}

  \end{enumerate}
  
\end{theorem}

\begin{example}
Consider $f=\mathbbm{1}_{(a,b)}$ for an interval $(a,b)\subset[0,1]$. $I_{\sigma,N}(f)$ is then simply the number of eigenvalues in the arc $[e^{2i\pi a}, e^{2i\pi b}]$. The function $f$ is obviously of bounded variation. 
This example has been treated in the simple case where $\theta=1$ in \cite{wieand}. We will see here that Theorem~\ref{Thm2} enables us to extend the results of \cite{wieand} to any value of $\theta>0$.
Indeed, the error in the composite trapezoidal approximation $R_j(f)$ is very easy to compute for an indicator function:
\begin{equation}
R_j(f)=\frac 1j\left(\{ja\}-\{jb\}\right).
\end{equation}
Obviously, in this case, Theorem~\ref{Thm2} applies and we have that
$$\frac{I_{\sigma, N}(f)-\mathbb{E}[I_{\sigma, N}(f)]}{\sqrt{C\theta\log N}}\stackrel{(d)}{\Rightarrow}\mathcal{N}(0,1).$$

We can also deduce the asymptotic behavior of the expectation and of the variance, using the conclusions of Theorem~\ref{Thm2} and the computations made in the particular  case $\theta=1$ in \cite{wieand}.
Indeed, it is shown in \cite{wieand}, that for a constant $c_1(a,b)$ 
\begin{equation}
\sum_{j=1}^NR_j(f)=  -c_1(a,b) \log N + o(\log N).
\end{equation}
So that from the statement proven in Theorem~\ref{Thm2} : 
\begin{equation}
\E[{I_{\sigma, N}}(f)]= N\int_0^1 f(x)dx+\theta \sum_{j=1}^NR_j(f) + O(1).
\end{equation}
We see that
\begin{equation}
\E[{I_{\sigma, N}}(f)]= N(b-a) - \theta c_1\log N + o(\log N).
\end{equation}
The value of $c_1(a,b)$ is studied in \cite{wieand}. It depends on the fact that a and b are rational or not. It vanishes if a and b are both irrational.

We also have, from the computations in \cite{wieand}, that there exists a positive constant $c_2(a,b)$ such that
\begin{equation}
\sum_{j=1}^NjR_j(f)^2= c_2(a,b) \log N + o(\log N).
\end{equation}
So that we have for any $\theta>0$, by Theorem~\ref{Thm2}, that
$$\Var{I_{\sigma, N}(f)}\sim c_2(a,b) \theta\log N.$$
The value of $c_2(a,b)$ also depends on the arithmetic properties of a and b, and is studied in \cite{wieand}.
\end{example}

\begin{remark}
We want to point out that $f$ being of bounded variation is not a necessary condition in order to get a Gaussian limit distribution. But, when $f$ is of bounded variation, it is easy to see that there exists a constant $C$ such that
\begin{equation}
\Var{I_{\sigma, N}(f)}\leq C\log N.
\end{equation}
The case treated in the example above gives the maximal normalization for functions of bounded variation. 
\end{remark}

The remainder of this article is organized as follows. In Section~2, we state our results with weaker assumptions than the theorems given in this introduction. These assumptions use the classical notion of Cesaro means of fractional order, which we recall in the first subsection of Section 2.  In Section 3, we prove the Corollaries \ref{cor1}, \ref{cor2} and  \ref{cor3}, using estimates on the trapezoidal approximation. In order to prove the main results of Section 2, our main tool will be the Feller coupling. This is natural since the problem is translated by the basic equality \eqref{RandI} in terms of cycle counts of random permutations. In Section 4, we will need to improve on the known bounds for the approximation given by this coupling (see for example \cite{arratiabarbourtavare2} or \cite{barbourtavare}) and relate these bounds to Cesaro means. We will then be ready to prove in Section 5 and Section 6 our general results as stated in Section 2 and that these more general results imply the two theorems of this introduction Theorems~\ref{Thm} and \ref{Thm2}.. Finally in the very short Section 7, we give an explicit expression for the expectation and variance of the linear statistics as promised in Remark \ref{expandvar}.

\section{Cesaro means and convergence of linear statistics}

\subsection{Cesaro Means} \ \newline

We will state here our optimal results in terms of convergence of the Cesaro means of fractional order.
First, we will need to recall the classical notion of Cesaro means of order $\theta$ and of Cesaro convergence $(C,\theta)$ for a sequence of real numbers, say $s= (s_j)_{j \geq 0}$ (see \cite{zygmund}, Volume 1, p. 77, formulae (1.14) and (1.15)).
\begin{definition}\label{cesaro_def}
\begin{enumerate}
\item[(i)] The Cesaro numbers of order $\alpha > -1$ are given by
\begin{equation}
A^{\alpha}_N := \binom{N+\alpha}{N}.
\end{equation}
\item[(ii)] The Cesaro mean of order $\theta>0$ of the sequence $s= (s_j)_{j \geq 0}$ is given by
\begin{equation}
\sigma_N^{\theta}(s) = \sum_{j=0}^{N} \frac{A^{\theta-1}_{N-j}}{A^{\theta}_{N}} s_j.
\end{equation}
\item[(iii)] A sequence  of real numbers $s=(s_j)_{j \geq 0}$ is said to be convergent in Cesaro sense of order $\theta$ (or in $(C,\theta)$ sense) to a limit $\ell$ iff the sequence of Cesaro means $\sigma_N^{\theta}(s)$ converges to $\ell$.
\end{enumerate}
\end{definition}

Let us recall the following basic facts about Cesaro convergence (see \cite{zygmund}):
\begin{lemma}\label{cesarofacts}

\begin{enumerate}
\item[(i)] Convergence in the $(C,\theta_1)$ sense to a limit $\ell$, implies convergence $(C,\theta_2)$ to the same limit for any $\theta_1 \leq \theta_2$.
\item[(ii)] Usual convergence is (C,0) convergence. The classical Cesaro convergence is (C,1) convergence.
\item[(iii)] If the sequence $(s_j)_{j \geq 0}$ is bounded and converges $(C,\theta_1)$ to a limit $\ell$ for some value $\theta_1>0$, then it converges $(C,\theta)$ to the same limit, for any $\theta>0$.
\end{enumerate}
\end{lemma}

These facts are all classical, see \cite{zygmund} for a proof, in particular Lemma (2.27), p. 70, Volume 2 for a proof of (iii).

\subsection{The case of bounded variance, non-Gaussian limits} \ \newline

We will give here a sharper statement than Theorem~\ref{Thm} and prove that it implies Theorem~\ref{Thm}.
Define the sequence $u(f)= (u_j(f))_{j\geq1}=(jR_j(f))_{j\geq1}$.

\begin{theorem}\label{Thm1general}
Let  $\theta$ be any positive number, and assume that the sequence $|u(f)|= (|u_j(f)|)_{j\geq1}$ converges to zero in the Cesaro $(C,  \theta)$ sense if $\theta < 1$.
Also assume that 
\begin{equation}\label{Hyp1}
\sum_{j=1}^{\infty}jR_j(f)^2 \in (0,\infty).
\end{equation}
 Then, 
\begin{enumerate}

\item  under the Ewens distribution $\nu_{N,\theta}$, the distribution of the centered linear statistic 
$$I_{\sigma,N}(f)-\mathbb{E}[I_{\sigma,N}(f)]$$
 converges weakly, as $N$ goes to infinity, to a non-Gaussian infinitely divisible distribution $\mu_{f,\theta}$.
 
 \item The  distribution $\mu_{f,\theta}$ is defined by its Fourier transform 

\begin{equation}\label{mufourier}\widehat{\mu}_{f,\theta}(t)=\exp\left(\theta\int(e^{itx}-1-itx)dM_f(x)\right),\end{equation}
where the L\'{e}vy measure $M_f$ is given by
 \begin{equation}\label{measuremu}
       M_f=\sum_{j=1}^\infty\frac 1j\delta_{jR_j(f)}.
  \end{equation} 
  
  \item The  asymptotic behavior of the expectation of the linear statistic is given by
  \begin{equation}
\E[{I_{\sigma, N}}(f)]= N\int_0^1 f(x)dx+\sum_{j=1}^NR_j(f) + o(1).
\end{equation}
Here, the second term $\sum_{j=1}^NR_j(f)$ may diverge, but not faster than logarithmically.
\begin{equation}
\sum_{j=1}^NR_j(f) = O(\sqrt{\log N})
\end{equation}
  
  \item If, on top of the preceding assumptions, one assumes that the sequence $u(f)^2= (u_j(f)^2)_{j\geq1}$ converges in Cesaro $(C,1\wedge \theta)$ sense, then the asymptotic behavior of the variance of the linear statistic is given by
  \begin{equation}
\Var[{I_{\sigma, N}}(f)]= \theta \sum_{j=1}^N jR_j(f)^2 +o(1).
\end{equation}

  \end{enumerate}
  
\end{theorem}

Theorem~\ref{Thm1general} will be proved in Section 5.

\subsection{The case of unbounded variance, Gaussian limits} \ \newline

We will give here a slightly sharper statement than Theorem~\ref{Thm2} and prove that it implies Theorem~\ref{Thm2}.

\begin{theorem}\label{Thm2general}
Let  $\theta$ be any positive number, and assume that 
\begin{equation}\label{Hyp2}
\sum_{j=1}^{\infty}jR_j(f)^2 = \infty
\end{equation}
and that 
\begin{equation}
\max_{1\leq j \leq N} |jR_j| = o( \eta_N),
\end{equation}
where $\eta_N^2= \theta \sum_{j=1}^{N}jR_j(f)^2 $. Then,
\begin{enumerate}

\item  under the Ewens distribution $\nu_{N,\theta}$, the distribution of the centered and normalized linear statistic 
\begin{equation}
\frac{I_{\sigma,N}(f)-\mathbb{E}[I_{\sigma,N}(f)]}{\sqrt{\Var{I_{\sigma, N}(f)}}}
\end{equation}
 converges weakly, as $N$ goes to infinity, to the Gaussian standard distribution $\mathcal{N}(0,1)$. 
  
  \item The asymptotic behavior of the expectation of the linear statistic is given by
  \begin{equation}
\E[{I_{\sigma, N}}(f)]= N\int_0^1 f(x)dx+\sum_{j=1}^NR_j(f)  + o(\eta_N).
\end{equation}
Here, the second term $\sum_{j=1}^NR_j(f)$ may diverge, but not faster than logarithmically.
 \begin{equation}
\sum_{j=1}^NR_j(f) = o(\eta_N \sqrt{\log N})
\end{equation}
  \item The asymptotic behavior of the variance of the linear statistic is given by
  \begin{equation}
\Var[{I_{\sigma, N}}(f)] \sim \eta_N^2= \theta \sum_{j=1}^N jR_j(f)^2 .
\end{equation}

  \end{enumerate}
  
\end{theorem}

This theorem will be proved in Section 6.

\section{Estimates on the trapezoidal rule and proofs of the Corollaries \ref{cor1}, \ref{cor2} and \ref{cor3}}

In this section, we will discuss known results about the quality of the composite trapezoidal approximation for periodic functions, in order to relate the decay of the $R_j(f)$'s to the regularity of $f$. Moreover, we will give proofs of Corollary~\ref{cor1}, Corollary~\ref{cor2}, Corollary~\ref{cor3}. \\

\subsection{Jackson-type estimates on the composite trapezoidal approximation} \ \newline

In order to relate the decay of the $R_j(f)$'s to the regularity of $f$, we can use two related approaches. First, we can control directly the size of the $R_j$'s by Jackson type inequalities as in \cite{buttgenbach}, \cite{uribe} or \cite{rahman}. Or we may use the Poisson summation formula given in \eqref{poissonsummation} and use the decay of the Fourier coefficients of $f$. \\\\
We start by using the first approach, and recall known Jackson-type estimates of the error in the trapezoidal approximation.
\begin{lemma}\label{Jackson}

\begin{enumerate}
\item[(i)] There exists a constant $C\leq 179/180$ such that
\begin{equation}\label{omega2}
|R_j(f)|\leq C\omega_2(f,1/2j),
\end{equation}
where 
\begin{equation}
\omega_2(f,\delta)=\sup_{|h|\leq\delta, x\in[0,1]}|f(x+2h)-2f(x+h)+f(x)|.
\end{equation}

\item[(ii)]  If the function $f$ is in $C^1$, then
\begin{equation}\label{omega1}
|R_j(f)|\leq C \frac{\omega(f',1/j)}{2j}.
\end{equation}

\item[(iii)] If the function $f$ is in $W^{1,p}$, then
\begin{equation}
|R_j(f)|\leq C\omega^{(p)}(f',1/j)\frac 1{j^{1-1/p}}.
\end{equation}
\end{enumerate}
\begin{proof}
The first item is well known, see (see \cite{buttgenbach}).

The second item is a consequence of the first, since by the Mean Value Theorem
\begin{equation}
\omega_2(f,\delta) \leq \delta \omega(f',2\delta).
\end{equation}

The third item is also an easy consequence of the first since
\begin{equation}
f(x+2h)-2f(x+h)+f(x) = \int_x^{x+h} (f'(t+h)-f'(t))dt.
\end{equation}
So that
\begin{equation}
|f(x+2h)-2f(x+h)+f(x)| \leq \left(\int_0^1|f'(t+h)-f'(t)|^p dt \right)^{\frac{1}{p}} h^{\frac{p-1}{p}},
\end{equation}
which shows that
\begin{equation}
\omega_2(f,\delta)\leq\omega^{(p)}(f', \delta)\delta^{1-1/p}.
\end{equation}
\end{proof}

\end{lemma}

\subsection{Proofs of Corollary \ref{cor1} and \ref{cor2} using Jackson bounds }
\begin{proof}[Proof of Corollary~\ref{cor1}]
We can  control the decay of the $R_j(f)$'s using the item (ii) of Lemma \ref{Jackson}, which implies that
\begin{equation}
j R_j(f)^2 \leq  C^2 \frac{\omega(f',1/j)^2}{4j}
\end{equation}
It is then clear that under the assumption \eqref{omegacond}, the series $\sum_{j=1}^\infty jR_j(f)^2$ is convergent.
But \eqref{omegacond} implies that the Fourier series of $f$ is absolutely convergent, so by the result mentioned above (\cite{loxtonsanders}, p. 260) it is true that \eqref{A6} implies that one of the $R_j(f)$'s is non zero. And thus, $\sum_{j=1}^\infty jR_j(f)^2 \in (0,\infty)$. If we add that $f$ is obviously of bounded variation, we have then checked the assumptions of Theorem~\ref{Thm} and thus, proved Corollary~\ref{cor1}.
\end{proof}

\begin{proof}[Proof of Corollary~\ref{cor2}]
We can  here control the decay of the $R_j(f)$'s using the item (iii) of Lemma \ref{Jackson}, and the assumption \eqref{cor2omega}, which imply that
\begin{equation}
|R_j(f)|\leq C\frac 1{j^{1+\alpha-1/p}}.
\end{equation}
So, if $\alpha>1/p$, the series $\sum_{j=1}^\infty jR_j(f)^2$ is convergent, since
\begin{equation}
jR_j^2(f)\leq\frac C{j^{1+2(\alpha-1/p)}}.
\end{equation}
Moreover, as above, it is easy to see that \eqref{cor2omega} implies that the Fourier series of $f$ is absolutely convergent, so by the result mentioned above (\cite{loxtonsanders}, p. 260) it is true that \eqref{A6} implies that one of the $R_j(f)$'s is non zero. Again, $f$ is obviously of bounded variation, we have then checked the assumptions of Theorem \ref{Thm} and thus, proved Corollary~\ref{cor2}.
\end{proof}

Remark: It is in fact true that $\lim_{j\rightarrow\infty} jR_j(f)=0$ is satisfied as soon as $f\in W^{1,p}$ (see \cite{uribe}).

\subsection{Proofs of Corollary \ref{cor3} and the Poisson summation formula} \ \newline

We now turn to the proof of Corollary~\ref{cor3}, using the second possible approach, i.e. the Poisson Summation Formula, \eqref{poissonsummation}.

\begin{proof}[Proof of Corollary~\ref{cor3}]
Let $f$ be in $H^s$, $s>1$ and consider its Fourier series
\begin{equation}
f(x)=a_0+\sum_{n=1}^\infty a_n\cos(n2\pi i x)+\sum_{n=1}^\infty b_n\sin(n2\pi i x).
\end{equation}
Then there exists a sequence  $(c_k)_{k\geq1}\in\ell^2$ such that
\begin{equation}
a_k=\frac{c_k}{k^s}.
\end{equation}
So,
\begin{equation}
C_j:=\sum_{\ell\geq1}\frac{c_{j\ell}}{\ell^s}
\end{equation}
 is in $\ell^2$ by Lemma 4 of \cite{rahman}, p. 131. Thus, using the Poisson summation formula \eqref{poissonsummation},
\begin{equation}
R_j(f)=\frac {C_j}{j^s},
\end{equation}
which is more than enough to prove that the series $\sum_{j=1}^\infty jR_j(f)^2$ is convergent.
Moreover, as above, it is easy to see that \eqref{A6} implies  that one of the $R_j(f)$'s is non zero, and that $f$ is obviously of bounded variation. We have then checked the assumptions of Theorem~\ref{Thm} and thus, proved Corollary~\ref{cor3}.
\end{proof}

\section{Bounds on the Feller coupling and Cesaro Means}

\subsection{The Feller Coupling} \ \newline

Let $\sigma\in\mathcal{S}_N$ be a given permutation and $\alpha_j(\sigma)$ be the number of $j$-cycles of $\sigma$. A classical result is that under the Ewens distribution $\nu_{N,\theta}$, the joint distribution of $(\alpha_1(\sigma),\dots,\alpha_N(\sigma))$ is given by 

\begin{equation}
\nu_{N,\theta}[(\alpha_1(\sigma),\dots,\alpha_N(\sigma))=(a_1,\dots,a_N)]=\mathbbm{1}_{\sum_{j=1}^Nja_j=N}\frac {N!}{\theta_{(N)}}\prod_{j=1}^N\left(\frac{\theta}{j}\right)^{a_j}\frac 1{a_j!}, 
\end{equation}
where $\theta_{(N)}=\theta(\theta+1)\dots(\theta+N-1)$.
\newline

We recall now the definition and some properties of the Feller coupling, a very useful tool to study the asymptotic behavior of $\alpha_j(\sigma)$ (see for example \cite{arratiabarbourtavare2}, p. 523). \\ Consider a probability space $(\Omega, \mathcal{F}, \mathbb{P})$ and a sequence $(\xi_i)_{i\geq 1}$ of independent Bernoulli random variables defined on $(\Omega, \mathcal{F})$ such that

$$\mathbb{P}[\xi_i=1]=\frac \theta{\theta+i-1}\quad\textrm{and}\quad\mathbb{P}[\xi_i=0]=\frac{i-1}{\theta+i-1}.$$ 
For $1\leq j\leq N$, denote the number of spacings of length $j$ in the sequence $1\xi_2\cdots \xi_N 1$ by $C_j(N)$, i.e.

\begin{equation}\label{C}
C_j(N)=\sum_{i=1}^{N-j}\xi_i(1-\xi_{i+1})\dots(1-\xi_{i+j-1})\xi_{i+j}+\xi_{N-j+1}(1-\xi_{N-j+2})\dots(1-\xi_N).
\end{equation}

Define $(W_{jm})_{j\geq1}$ by 

\begin{equation}\label{W}
W_{jm}=\sum_{i=m+1}^\infty\xi_i(1-\xi_{i+1})\dots(1-\xi_{i+j-1})\xi_{i+j}
\end{equation}
and set for $j\geq1$,
\begin{equation}
W_j:=W_{j0}.
\end{equation}

Define 
\begin{equation}J_N=\min\{j\geq1: \xi_{N-j+1}=1\}\end{equation}
and
\begin{equation}K_N=\min\{j\geq1:\xi_{N+j}=1\}.\end{equation}

With the notations above, we state the following result of \cite{barbourtavare}, p.169:\\\\
\begin{theorem}\label{theorembarbourtavare}  
Under the Ewens distribution $\nu_{N,\theta}$,
\begin{enumerate}
\item[(i)]$(C_j(N))_{1\leq j\leq N}$ has the same distribution as $(\alpha_j(\sigma))_{1\leq j\leq N}$ , i.e. for any $a=(a_1,\dots,a_N)\in\mathbb{N}^N$,
\begin{equation}
\mathbb{P}[(C_1(N),\dots,C_N(N)=a]=\nu_{N,\theta}[(\alpha_1(\sigma),\dots,\alpha_N(\sigma)=a],
\end{equation}

\item[(ii)] $(W_j)_{1\leq j\leq N}$ are independent Poisson random variables with mean $\theta/j$,
\item[(iii)] and
\begin{equation}\label{estimatedistance}
     |C_j(N)-W_j|\leq W_{jN}+\mathbbm{1}_{\{J_N+K_N=j+1\}}+\mathbbm{1}_{\{J_N=j\}}.
\end{equation}
\end{enumerate}

\end{theorem}

We will need to improve on the known results for the Feller coupling.
In particular we will need the following. For any sequence of real numbers $(u_j)_{j\geq1}$, define
\begin{equation}
G_N= \sum_{j=1}^N u_jC_j(N)  
\end{equation} 
and
\begin{equation}
H_N =\sum_{j=1}^N u_j W_j.
\end{equation}
We will need to control the $L^1$ and $L^2$-distances between the random variables $G_N$ and $H_N$.
In order to prove Theorem~\ref{Thm} and Theorem~\ref{Thm2}, we will apply these estimates to the case where
the sequence $u_j$ is chosen to be $u_j(f)= jR_j(f)$.

\subsection{$L^1$ bounds on the Feller Coupling} \ \newline

We begin with the control of the $L^1$-distance in this subsection.
We first state our result in a very simple (but not optimal) shape.

\begin{lemma}\label{L1boundwithmax}
For every $\theta>0$, there exists a constant $C(\theta)$ such that, for every integer $N$,
\begin{equation}
\mathbb{E}(|G_N-H_N|) \leq C(\theta) \max_{1\leq j\leq N} |u_j| .
\end{equation}
\end{lemma}

This result is a trivial consequence of a deeper result, that we now give after introducing some needed notations.
We recall that for any real number $x$ and integer $k$,
\begin{equation}
\binom{x}{k}=\frac{x(x-1)\dots(x-k+1)}{k!}\quad.
\end{equation}
We now define for any $\theta>0$ and every $1\leq j\leq N$, 
\begin{equation}\label{definition_psi}
\Psi_N(j):=\binom{N-j+\gamma}{N-j}\binom{N+\gamma}{N}^{-1}=\prod_{k=0}^{j-1}\frac{N-k}{\theta+N-k-1},
\end{equation}
where $\gamma = \theta -1$.

We then have: 

\begin{lemma}\label{asymplk1}
\begin{equation}
\mathbb{E}|G_N-H_N| \leq  \frac{C(\theta)}{N} \sum_{j=1}^N |u_j|+ \frac{\theta}{N} \sum_{j=1}^N |u_j|\Psi_N(j) 
\end{equation}
\end{lemma}

Lemma \ref{asymplk1} is obviously a direct consequence of the following:
 \begin{lemma}\label{asymplk}
Let $\theta>0$,  there exists a constant $C(\theta)$, such that, for every $1\leq j\leq N$
\begin{equation}
\mathbb{E}|C_j(N)-W_j| \leq  \frac{C(\theta)}{N} + \frac{\theta}{N} \Psi_N(j) 
\end{equation}
\end{lemma}

In order to prove Lemma~\ref{asymplk}, we note that, by \ref{estimatedistance}, 
\begin{equation}
\mathbb{E}|C_j(N)-W_j| \leq \mathbb{E}(W_{jN})+ \mathbb{P}(J_N+K_N=j+1)+\mathbb{P}(J_N=j).
\end{equation}
It thus suffices to provide bounds
on $\mathbb{E}[W_{jN}]$, $\mathbb{P}[J_N=j]$ and $\mathbb{P}[J_N+K_N=j+1]$.

\begin{lemma}
For any $\theta>0$ and for every $1\leq j\leq N$,
\begin{equation}\label{bdE}
   \mathbb{E}(W_{jN})\leq\frac{\theta^2}{N-1}.
\end{equation}

\begin{proof}
Let 
\begin{equation}\label{definition_U}
U_i^{(j)}:=\xi_i(1-\xi_{i+1})\dots(1-\xi_{i+j-1})\xi_{i+j},
\end{equation}
then, for $i\geq2$, 
\begin{equation}\label{expectation_U}
\mathbb{E}(U_i^{(j)})\leq \mathbb{E}(\xi_i)\mathbb{E}(\xi_{i+j})=\frac{\theta^2}{(\theta+i-1)(\theta+i+j-1)}\leq \frac{\theta^2}{(i-1)^2}.
\end{equation}
By \eqref{expectation_U}, we have immediately that, for any $\theta>0$,
\begin{equation}
\mathbb{E}(W_{jN})=\sum_{i=N+1}^\infty U_i^{(j)}\leq \theta^2\sum_{\ell=N}^\infty\frac 1{\ell^2}\leq \frac{\theta^2}{N-1}.
\end{equation}

\end{proof}
\end{lemma}

We  compute next  the distribution of the random variable $J_N$ explicitly.

\begin{lemma}\label{P=Psi}
\begin{eqnarray}
\mathbb{P}[J_N=j]=\frac \theta N\Psi_N(j).
\end{eqnarray}

\begin{proof}
The random variable $J_N$ is equal to $j$ if and only if $\xi_N=0, \xi_{N-1}=0, \dots, \xi_{N-j+2}=0$ and $\xi_{N-j+1}=1$. So, for any $1\leq j\leq N$,
\begin{eqnarray}
\mathbb{P}[J_N=j]&=&\frac{N-1}{\theta+N-1}\times\frac{N-2}{\theta+N-2}\dots\frac{N-(j-1)}{\theta+N-(j-1)}\times\frac{\theta}{\theta+N-j}\nonumber\\
&=&\label{pj1}\frac\theta N\prod_{k=0}^{j-1}\frac{N-k}{\theta+N-k-1}=\frac \theta N\Psi_N(j),
\end{eqnarray}
which proves the claim.
\end{proof}
\end{lemma}

We now bound the distribution of the random variable $J_N+K_N$.


\begin{lemma}
For any $\theta>0$, 

\begin{equation}\label{bdP5}
\mathbb{P}[K_N+J_N=j+1]\leq\frac{\theta}{N}.
\end{equation}

\begin{proof}

Consider first the random variable $K_N$. For any $\theta>0$,  
\begin{eqnarray}\label{kn}
\mathbb{P}[K_N=j]&=&\frac{N}{\theta+N}\times\frac{N+1}{\theta+N+1}\dots\frac{N+j-2}{\theta+N+j-2}\times\frac{\theta}{\theta+N+j-1}\nonumber\\
&\leq&\frac{N\theta}{(\theta+N+j-2)(\theta+N+j-1)}\nonumber\\ &\leq&\frac{\theta}{\theta+N+j-1}\leq\frac{\theta}{N+\theta}.
\end{eqnarray}


For any $\theta>0$, use \eqref{kn} to write
\begin{eqnarray}\label{bdP4}
\mathbb{P}[K_N+J_N=j+1]&=&\sum_{\ell=1}^j\mathbb{P}(J_N=j+1-\ell)\mathbb{P}(K_N=\ell)\nonumber\\ &\leq&\frac{\theta}{N+\theta}\sum_{\ell=1}^j\mathbb{P}(J_N=j+1-\ell)\nonumber\\&=&\frac{\theta}{N+\theta}\mathbb{P}(J_N\leq j)\leq\frac{\theta}{N}.
\end{eqnarray}
\end{proof}
\end{lemma}

The last three lemmas imply the result of Lemma \ref{asymplk}.
We have now controlled the $L^1$-distance between $G_N$ and $H_N$.

\subsection{$L^2$ bounds on the Feller coupling} \ \newline

We now turn to the control of the $L^2$-distance between the random variables $G_N$ and $H_N$.
We first state our result in a simple (but not optimal) shape.

\begin{lemma}\label{L2boundwithmax}
For every $\theta>0$, there exists a constant $C(\theta)$ such that, for every integer $N$,
\begin{equation}
\mathbb{E}\left((G_N-H_N)^2\right) \leq C(\theta) \max_{1\leq j\leq N} |u_j| ^2
\end{equation}
\end{lemma}

This result is an immediate consequence of the following much more precise statement.

\begin{lemma}\label{Lemma_G-Hsquared}
For every $\theta>0$, there exists a constant $C(\theta)$ such that, for every integer $N$,
\begin{eqnarray}
\mathbb{E}\left((G_N-H_N)^2\right)&\leq& C(\theta)  \left[( \frac 1N\sum_{j=1}^N|u_j|)^2+ \frac1N\sum_{j=1}^N|u_j|^2\right.\nonumber\\
&+&\frac{1}{N^2}\sum_{j=1}^N|u_j|  \sum_{k=1}^N|u_k| \Psi_N(k) \nonumber\\
&+& \left.\frac{1}{N}\sum_{j=1}^N|u_j|^2\Psi_N(j) \right]
\end{eqnarray}

\begin{proof}
We note that
\begin{equation}
\mathbb{E}\left((G_N-H_N)^2\right)\leq \sum_{j,k=1}^N|u_j||u_k|\mathbb{E}(|C_j-W_j||C_k-W_k|).
\end{equation}

By \eqref{estimatedistance}, for any fixed $1\leq j,k\leq N$,
\begin{eqnarray}\label{C-Wsquared}
|C_j-W_j||C_k-W_k|&\leq& W_{j,N}W_{k,N}+W_{j,N}\mathbbm{1}_{J_N=k}+W_{j,N}\mathbbm{1}_{J_N+K_N=k+1}\nonumber\\
&&+\mathbbm{1}_{J_N+K_N=j+1}W_{k,N}+\mathbbm{1}_{J_N+K_N=j+1}\mathbbm{1}_{J_N=k}\nonumber\\
&&+\mathbbm{1}_{J_N+K_N=j+1}\mathbbm{1}_{J_N+K_N=k+1}\nonumber\\
&&+\mathbbm{1}_{J_N=j}W_{k,N}+\mathbbm{1}_{J_N=j}\mathbbm{1}_{J_N=k}\nonumber\\
&&+\mathbbm{1}_{J_N=j}\mathbbm{1}_{J_N+K_N=k+1}.
\end{eqnarray}

To control \eqref{C-Wsquared}, we will give upper bounds for all the terms on the RHS. We start by giving a bound for $\mathbb{E}(W_{j,N}W_{k,N})$:
By \eqref{definition_U}, we have

\begin{eqnarray}\label{expectation_WW}
\mathbb{E}(W_{j,N}W_{k,N})&=&\sum_{i,\ell\geq N+1}U_i^{(j)}U_\ell^{(k)}\nonumber\\
&=&\sum_{\substack{i,\ell\geq N+1\\i<\ell}}U_i^{(j)}U_\ell^{(k)}+\sum_{\substack{i,\ell\geq N+1\\i>\ell}}U_i^{(j)}U_\ell^{(k)}+\sum_{i\geq N+1} U_i^{(j)}U_i^{(k)}.\nonumber\\
\end{eqnarray}

We write the first term on the RHS as follows:
\begin{equation}
\sum_{\substack{i,\ell\geq N+1\\i<\ell<i+j}}U_i^{(j)}U_\ell^{(k)}+\sum_{\substack{i,\ell\geq N+1\\i+j<\ell}}U_i^{(j)}U_\ell^{(k)}+\sum_{i\geq N+1}U_i^{(j)}U_{i+j}^{(k)}.
\end{equation}

It is easy to see that for any $\ell\in(i,i+j)$, $U_i^{(j)}U_\ell^{(k)}=0$. If $\ell$ is strictly larger than $i+j$, then $U_i^{(j)}$ and $U_\ell^{(k)}$ are independent. This gives, using \eqref{expectation_U},

\begin{eqnarray}
\sum_{\substack{i,\ell\geq N+1\\i+j<\ell}}\mathbb{E}\left(U_i^{(j)}U_\ell^{(k)}\right)&=&\sum_{\substack{i,\ell\geq N+1\\i+j<\ell}}\mathbb{E}\left(U_i^{(j)}\right)\mathbb{E}\left(U_\ell^{(k)}\right)\nonumber\\
&\leq&\sum_{i=N+1}^\infty\mathbb{E}\left(U_i^{(j)}\right)\sum_{\ell>i+j}\frac{\theta^2}{(\ell-1)^2}\nonumber\\
&\leq&\sum_{i=N+1}^\infty \frac{\theta^4}{(i-1)^2}\frac{1}{(i+j-1)}\leq \frac{C(\theta)}{N^2}.
\end{eqnarray}

Also, by the same argument,
\begin{eqnarray}
\sum_{i=N+1}^\infty U_i^{(j)}U_{i+j}^{(k)}&\leq& \sum_{i=N+1}^\infty\mathbb{E}(\xi_i)\mathbb{E}(\xi_{i+j})\mathbb{E}(\xi_{i+j+k})\nonumber\\
&\leq& \sum_{i=N+1}^\infty\frac{\theta^3}{(i-1)(i+j-1)(i+j+k-1)}\leq \frac{C(\theta)}{N^2}.\qquad
\end{eqnarray}

For the bound of \eqref{expectation_WW}, we consider now the second and the third term on the RHS. But the second term can be bounded similarly to the first term. For the third term in \eqref{expectation_WW}, we observe that $U_i^{(j)}U_i^{(k)}=0$ if $j\neq k$. So, by \eqref{expectation_U}

\begin{eqnarray}
\sum_{i\geq N+1} \mathbb{E}\left(U_i^{(j)}U_i^{(k)}\right)&=&\sum_{i\geq N+1} \mathbb{E}\left(\left(U_i^{(j)}\right)^2\right)=\sum_{i\geq N+1} \mathbb{E}\left(U_i^{(j)}\right)\nonumber\\
&\leq&\sum_{i\geq N+1} \frac{\theta}{(i-1)^2}\leq \frac{C(\theta)}{N}.
\end{eqnarray}

This gives
\begin{equation}
\mathbb{E}(W_{j,N}W_{k,N})=\left\{\begin{array}{cc} C(\theta)/N^2 & \textrm{if }j\neq k\\ C(\theta)/N & \textrm{if }j=k.\end{array}\right. 
\end{equation}

So,
\begin{eqnarray}
\sum_{j,k=1}^N|u_j||u_k|\mathbb{E}(W_{j,N}W_{k,N})&\leq& C_1(\theta)\left(\frac 1N\sum_{j=1}^N|u_j|\right)^2+C_2(\theta)\frac1N\sum_{j=1}^N|u_j|^2\nonumber\\
\end{eqnarray}

Obviously, $W_{j,N}$ and $\mathbbm{1}_{J_N=k}$ are independent. So, the expectation of the second term on the RHS in \eqref{C-Wsquared} is bounded as follows:
\begin{equation}
\mathbb{E}(W_{j,N}\mathbbm{1}_{J_N=k})\leq \frac{C(\theta)}{N}\mathbb{P}(J_N=k).
\end{equation}

Of course, this bound is also valid for $\mathbb{E}(\mathbbm{1}_{J_N=j}W_{k,N})$.\\

Then,
\begin{eqnarray}
\sum_{j,k=1}^N|u_j||u_k|\mathbb{E}(W_{j,N}\mathbbm{1}_{J_N=k})&\leq& \frac{C(\theta)}{N}\sum_{j=1}^N|u_j|\cdot \sum_{k=1}^N|u_k|\mathbb{P}.(J_N=k)\nonumber\\
\end{eqnarray}

For $W_{j,N}\mathbbm{1}_{J_N+K_N=k+1}$, we write
\begin{eqnarray}
&&\mathbb{E}(W_{j,N}\mathbbm{1}_{J_N+K_N=k+1})=\mathbb{E}\left(\sum_{\ell=1}^k W_{j,N}\mathbbm{1}_{J_N+K_N=k+1}\mathbbm{1}_{J_N=\ell}\right)\nonumber\\
&&\qquad=\sum_{\ell=1}^k \mathbb{E}(W_{j,N+k-\ell}\mathbbm{1}_{K_N=k+1-\ell})\mathbb{P}(J_N=\ell).\qquad
\end{eqnarray}

But,
\begin{eqnarray}
&&\mathbb{E}(W_{j,N+k-\ell}\mathbbm{1}_{K_N=k+1-\ell})\nonumber\\
&&\quad=\mathbb{E}\left(\sum_{i>N+k+1-\ell}U_i^{(j)}\mathbbm{1}_{K_N=k+1-\ell}\right)+\mathbb{E}\left(U_{N+k+1-\ell}^{(j)}\mathbbm{1}_{K_N=k+1-\ell}\right),\qquad\qquad
\end{eqnarray}

where $\sum_{i>N+k+1-\ell}U_i^{(j)}$ and $\mathbbm{1}_{K_N=k+1-\ell}$ are independent and 
\begin{eqnarray}
&&\mathbb{E}\left(U_{N+k+1-\ell}^{(j)}\mathbbm{1}_{K_N=k+1-\ell}\right)\nonumber\\
&&\qquad=\mathbb{E}\left(\xi_{N+k+1-\ell}(1-\xi_{N+k+2-\ell})\dots(1-\xi_{N+k+j-\ell})\xi_{N+k+j+1-\ell}\xi_{N+k+1-\ell}\right)\nonumber\\
&&\qquad=\mathbb{E}\left(U_{N+k+1-\ell}^{(j)}\right)\leq \frac{C(\theta)}{(N+k-\ell)^2}\leq \frac{C(\theta)}{N^2}.
\end{eqnarray}

So, by \eqref{kn}
\begin{eqnarray}
\mathbb{E}(W_{j,N+k-\ell}\mathbbm{1}_{K_N=k+1-\ell})&\leq& \frac{C_1(\theta)}{N}\mathbb{P}(K_N=k+1-\ell)+\frac{C_2(\theta)}{N^2}\leq\frac{C(\theta)}{N^2},\qquad\qquad
\end{eqnarray}

which gives 
\begin{equation}
\mathbb{E}(W_{j,N+k-\ell}\mathbbm{1}_{K_N=k+1-\ell})\leq \frac{C(\theta)}{N^2}.
\end{equation}
This gives also the bound for $\mathbb{E}(\mathbbm{1}_{J_N+K_N=j+1}W_{k,N})$. \\

Then,
\begin{eqnarray}
\sum_{j,k=1}^N|u_j||u_k|\mathbb{E}(W_{j,N}\mathbbm{1}_{J_N+K_N=k+1})&\leq& C(\theta)\left(\frac 1N\sum_{j=1}^N|u_j|\right)^2.
\end{eqnarray}
For the remaining terms in \eqref{C-Wsquared}, we observe that 

\begin{eqnarray}
\mathbb{E}(\mathbbm{1}_{J_N+K_N=j+1}\mathbbm{1}_{J_N=k})&=&\mathbb{E}(\mathbbm{1}_{K_N=j+1-k}\mathbbm{1}_{J_N=k})=\mathbb{P}(K_N=j+1-k)\mathbb{P}(J_N=k)\nonumber\\
&\leq&\frac{C(\theta)}{N}\mathbb{P}(J_N=k).
\end{eqnarray}

This applies for $\mathbbm{1}_{J_N=j}\mathbbm{1}_{J_N+K_N=k+1}$, as well.\\

Then
\begin{eqnarray}
\sum_{j,k=1}^N|u_j||u_k|\mathbb{E}(\mathbbm{1}_{J_N+K_N=j+1}\mathbbm{1}_{J_N=k})&\leq& \frac{C(\theta)}{N}\sum_{j=1}^N|u_j|\cdot \sum_{k=1}^N|u_k|\mathbb{P}(J_N=k)\nonumber\\
\end{eqnarray}

Also,
\begin{equation}
\mathbbm{1}_{J_N+K_N=j+1}\mathbbm{1}_{J_N+K_N=k+1}=\left\{\begin{array}{cc}1 & \textrm{if }k=j\\ 0 & \textrm{if }k\neq j,\end{array}\right.
\end{equation}
so, 

\begin{equation}
\mathbb{E}(\mathbbm{1}_{J_N+K_N=j+1}\mathbbm{1}_{J_N+K_N=k+1})=\left\{\begin{array}{cc}\mathbb{P}(J_N+K_N=j+1) & \textrm{if }k=j\\ 0 & \textrm{if }k\neq j\end{array}\right.
\end{equation}

and by \eqref{bdP5}
\begin{eqnarray}
\sum_{j,k=1}^N|u_j||u_k|\mathbb{E}(\mathbbm{1}_{J_N+K_N=j+1}\mathbbm{1}_{J_N+K_N=k+1})&=&\sum_{j=1}^N|u_j|^2\mathbb{P}(J_N+K_N=j+1)\nonumber\\
&\leq &\frac{C(\theta)}{N} \sum_{j=1}^N|u_j|^2.
\end{eqnarray}

It is obvious that $\mathbbm{1}_{J_N=j}\mathbbm{1}_{J_N=k}=0$ for $k\neq j$. So,
\begin{equation}
\mathbb{E}(\mathbbm{1}_{J_N=j}\mathbbm{1}_{J_N=k})=\left\{\begin{array}{cc}\mathbb{P}(J_N=j)& \textrm{if }k=j\\ 0 & \textrm{otherwise.}\end{array}\right.
\end{equation}

Then,
\begin{eqnarray}
\sum_{j,k=1}^N|u_j||u_k|\mathbb{E}(\mathbbm{1}_{J_N=j}\mathbbm{1}_{J_N=k})\leq \sum_{j=1}^N|u_j|^2\mathbb{P}(J_N=j),\end{eqnarray}
which, using also Lemma \ref{P=Psi}, proves the claim of Lemma \ref{Lemma_G-Hsquared}.
\end{proof}
\end{lemma}

\subsection{Cesaro Means and the Feller Coupling Bounds} \ \newline

The link between our estimates and Cesaro means of fractional order is given by an interesting interpretation of Cesaro means of order $\theta$  in terms of the random variable $J_N$.
\begin{lemma}\label{cesaroandJ}
The Cesaro mean $\sigma_N^{\theta}$ of order $\theta$ of a sequence $s=(s_j)_{j \geq 0}$, with $s_0=0$, is given by
\begin{equation}
\sigma_N^{\theta}(s) = \frac{N}{N+\theta} \sum_{j=1}^{N} s_j \mathbb{P}[J_N=j] =  \frac{\theta}{N+\theta} \sum_{j=1}^{N} s_j \Psi_N(j) 
\end{equation}
\end{lemma}

The proof of this lemma is immediate from Lemma \ref{P=Psi}, the Definition~\ref{cesaro_def} of the Cesaro means and of the numbers $\Psi_N(j)$, given in \eqref{definition_psi}.

Using this interpretation of the Cesaro means, we can state our results about the $L^1$ and $L^2$ distance between the variables $G_N$ and $H_N$ given in Lemma~\ref{asymplk1} and Lemma~\ref{Lemma_G-Hsquared} in terms of the Cesaro means of the sequence $u_j(f)$ and $u_j(f)^2$.
\begin{theorem}\label{cesaro1bound}
For an $\theta>0$, there exists a constant $C(\theta)$ such that
\begin{enumerate}
\item[(i)]
\begin{equation}
\mathbb{E}|G_N-H_N| \leq  C(\theta) (  \sigma_N^{1}(|u|) + \sigma_N^{\theta}(|u|) )
\end{equation}

\item[(ii)] and
\begin{equation}\label{cesaro2bound}
\mathbb{E}\left((G_N-H_N)^2\right) \leq C(\theta)  [ \sigma_N^1(|u|)^2+ \sigma_N^1(u^2)
+ \sigma_N^1(|u|)\sigma_N^{\theta}(|u|)+ \sigma_N^{\theta}(u^2) ].
\end{equation}
\end{enumerate}
\end{theorem}

This theorem is simply a rewriting of Lemma~\ref{asymplk1}, and Lemma~\ref{Lemma_G-Hsquared}, using the identification given in Lemma~\ref{cesaroandJ}.
It implies easily the following results

\begin{theorem}\label{CesaroL1}
If the sequence $(u_j)_{j \geq 1}$ converges in Cesaro $(C,\theta \wedge1)$ sense to 0, then 
\begin{equation}
\lim_{N \to \infty} \mathbb{E}|G_N-H_N| = 0.
\end{equation}

\begin{proof}
By assumption, the sequence converges in $(C,1)$ and in $(C,\theta)$ sense to 0.
Thus, the RHS of the bound given in Theorem~\ref{cesaro1bound} tends to zero, which proves Theorem \ref{CesaroL1}.

\end{proof}
\end{theorem} 

Similarly we can get the following result about convergence in $L^2$.

\begin{theorem}\label{CesaroL2}
If the sequences $(|u_j|)_{j \geq 1}$ and $(u_j^2)_{j \geq 1}$ both converge to zero in Cesaro $(C,\theta \wedge1)$, then 
\begin{equation}
\lim_{N \to \infty} \mathbb{E}((G_N-H_N)^2) = 0
\end{equation}

\begin{proof}
By assumption, the sequences  $(|u_j|)_{j \geq 1}$ and $(u_j^2)_{j \geq 1}$ converge in $(C,1)$ and in $(C,\theta)$ sense to 0.
Thus, the RHS of the bound given in \eqref{cesaro2bound} tends to zero, which proves Theorem \ref{CesaroL2}.
\end{proof}
\end{theorem}


\section{Proof of Theorem \ref{Thm1general} and Theorem \ref{Thm}}

\subsection{A simple convergence result for series of Poisson random variables} \ \newline

We give here a result of convergence in distribution for the random variables
\begin{equation}\label{H}H_N(f)=\sum_{j=1}^NW_ju_j(f),\end{equation} 
to an infinitely divisible law. This result is elementary since it only uses the fact that the random variables $W_j$'s are independent and Poisson.

\begin{lemma}\label{lk}
    Under the assumption \eqref{Hyp1}, i.e 
    \begin{equation}
    \sum_{j=1}^N jR_j^2 \in (0,\infty),
    \end{equation}
    the distribution $\mu_N$ of $H_N-\mathbb{E}[H_N]$ 
    converges weakly to the distribution $\mu_{f,\theta}$ defined by \eqref{mufourier}.

    \begin{proof} 
    The Fourier transform of $H_N-\mathbb{E}[H_N]$ is easy to compute, indeed:

    \begin{eqnarray}
       \log\widehat{\mu}_N(t)=\log\mathbb{E}\left[e^{it(H_N-\mathbb{E}[H_N])}\right]&=&\log\prod_{j=1}^N\mathbb{E}\left[\exp\left(itu_j\left(W_j-\frac \theta j\right)\right)\right]\nonumber\\ 
        &=& \sum_{j=1}^N\frac\theta j(e^{itu_j}-itu_j-1).
    \end{eqnarray}
  
  Obviously, for $|t|\leq T$, 
  
  \begin{equation}\left|\frac\theta j\left(e^{itu_j}-itu_j-1\right)\right|\leq\frac \theta j\frac{t^2u_j^2}{2}\leq\theta\frac{T^2}2\frac{u_j^2}{j}.\end{equation}
 By \ref{Hyp}, $\log\widehat{\mu}_N(t)$ converges absolutely uniformly and its limit
  
 \begin{equation}\psi(t)= \sum_{j=1}^\infty\frac\theta j(e^{itu_j}-itu_j-1)\end{equation}
 
 is continuous. By L\'{e}vy's Theorem, $\exp(\psi(t))$ is the Fourier transform of the probability measure $\mu_{f,\theta}$ and $\mu_N$ converges in distribution to $\mu_{f,\theta}$ as $N$ goes to infinity.  
 \end{proof}
\end{lemma}

Obviously, $\mu_{f,\theta}$ is an infinitely divisible distribution and its L\'{e}vy-Khintchine representation is easy to write. We recall that an infinitely divisible distribution $\mu$ has L\'{e}vy-Khintchine representation $(a, M, \sigma^2)$ if its Fourier transform is given by 
\begin{equation}\label{levykhintchine}
\widehat{\mu}(t)=\exp\left(\int \left(e^{itx}-1-\frac{itx}{1+x^2}\right)dM(x)+iat-\frac 12 \sigma^2t^2\right),
\end{equation}
where $a\in\mathbb{R}$, $\sigma>0$ and $M$ is an admissible Levy measure, i.e. $$\int \frac{x^2}{1+x^2}dM(x)<\infty.$$

The distribution $\mu_{f,\theta}$ in Lemma~\ref{lk} has therefore a L\'{e}vy-Khintchine representation $(a, \theta M, 0)$ with
\begin{equation} a=\int \left(\frac{x}{1+x^2}-x\right)dM(x)=\sum_{j=1}^\infty\lambda_j\left(\frac{u_j}{1+u_j^2}-u_j\right)\end{equation} and \begin{equation}M=\sum_{j=1}^\infty \frac1j\delta_{u_j}.\end{equation}
It is easy to see that the assumption \ref{Hyp} implies that $\int x^2dM(x)<\infty$ so that $M$ is admissible.

\subsection{Proof of Theorem \ref{Thm1general}} \ \newline

We proceed now to the proof of Theorem \ref{Thm1general}, by using the Feller coupling bounds proved in Section 4. \\

We first prove the first and second statements of Theorem \ref{Thm1general}.
Under the assumption that $\sum_{j=1}^{\infty} jR_J^2 < \infty$, we have seen that the sequence $|u_j|$ converges in $(C,1)$ sense to zero. Moreover, if $\theta <1$, we assumed in Theorem~\ref{Thm1general} that the sequence $|u_j|$ converges in $(C,\theta)$ sense to zero. Thus, we know that the assumption of Theorem~\ref{CesaroL1} is satisfied, and thus that, 
\begin{equation}\label{justela2}
\lim_{N \to \infty} \mathbb{E}|G_N-H_N| = 0.
\end{equation}
Using now Lemma~\ref{lk}, we have proved that $G_N - \mathbb{E}(G_N)$ converges in distribution to $\mu_{f,\theta}$ defined by \eqref{mufourier}.
But, by the basic identity \eqref{RandI}, we know that $I_{\sigma,N}(f)-\mathbb{E}[I_{\sigma,N}(f)]$ has the same distribution as $G_N(f)-\mathbb{E}[G_N(f)]$. This proves the first two statements of Theorem~\ref{Thm1general}.\\

The third statement is simple. Indeed, by \eqref{RandI} and by \eqref{justela2},
\begin{equation}
\mathbb{E}[I_{\sigma,N}(f)] = N \int_0^1 f(x)dx + \mathbb{E}(G_N)= N \int_0^1 f(x)dx + \mathbb{E}(H_N) +o(1).
\end{equation}
In order to complete the proof, it suffices to mention that the expectation of $H_N$ is easy to compute:
\begin{equation}\label{meanH}
\mathbb{E}(H_N) = \theta \sum_{j=1}^N \frac{u_j}{j} = \theta \sum_{j=1}^N R_j.
\end{equation}
This proves the third statement of Theorem~\ref{Thm1general}.\\

The proof of the fourth statement follows a similar pattern. Again, by \eqref{RandI}, 
\begin{equation}
\Var[I_{\sigma,N}(f)] =\Var(G_N).
\end{equation}
But if one also assumes, as in the fourth item of Theorem~\ref{Thm1general}, that the sequence $(u_j^2)$ converges in $(C,1 \wedge \theta)$ sense to zero, then by \ref{CesaroL2}
we know that
\begin{equation}
\lim_{N \to \infty} \mathbb{E}((G_N-H_N)^2) = 0.
\end{equation}
This, and \ref{justela2}, imply that
\begin{equation}
\Var(G_N) = \Var(H_N) + o(1)
\end{equation}
In order to complete the proof, it suffices to compute the variance of $H_N$:
\begin{equation}
\Var (H_N) = \theta \sum_{j=1}^N \frac{u_j^2}{j} = \theta \sum_{j=1}^N jR_j^2
\end{equation}
This proves the fourth statement  and completes the proof of of Theorem \ref{Thm1general}.

\subsection{Proof of Theorem \ref{Thm}}\ \newline

We show here how Theorem \ref{Thm1general} implies Theorem~\ref{Thm}.

We will need the following simple facts.
\begin{lemma}\label{simplefacts} 
\begin{enumerate}
\item[(i)]
The assumption $\sum_{j=1}^{\infty}jR_j(f)^2 < \infty$ implies the $(C,1)$ convergence of the sequence $(|u_j(f)|)_{j\geq1}$ to zero.

\item[(ii)]
If one assumes that $\sum_{j=1}^{\infty}jR_j(f)^2 < \infty$ and that the function $f$ is of bounded variation, then the sequence $(|u_j(f)|)_{j\geq1}$ converges in $(C,\theta)$ sense to zero, for any $\theta>0$. 
\item[(iii)]
If one assumes that $\sum_{j=1}^{\infty}jR_j(f)^2 < \infty$ and that the function $f$ is of bounded variation, then the sequence $(u_j(f)^2)_{j\geq1}$ converges in $(C,\theta)$ sense to zero, for any $\theta>0$. 
\end{enumerate}
\begin{proof}

The first item is well known (see statement (a), p. 79 of \cite{zygmund}, Volume 1). It is a consequence of the simple application of the Cauchy-Schwarz inequality
\begin{equation}
|\frac1N \sum_{j=1}^{N} u_j| \leq \frac1N (\sum_{j=1}^{N}jR_j^2)^{\frac12}(\sum_{j=1}^{N}j)^{\frac12} \leq (\sum_{j=1}^{\infty}jR_j^2)^{\frac12}.
\end{equation}
So that 
\begin{equation}\label{justela}
\limsup_{N\to\infty} |\frac1N \sum_{j=1}^{N} u_j| \leq  (\sum_{j=1}^{\infty}jR_j^2)^{\frac12}.
\end{equation}
But the LHS of \eqref{justela} does not depend on the initial $k$ values of the sequence $u_j$. By setting these $k$ values to zero, and by taking $k$ large enough, we can then make the RHS as small as we want. This implies that 
\begin{equation}
\lim_{N\to\infty} \frac1N \sum_{j=1}^{N} u_j=0.
\end{equation}
This is the $(C,1)$ convergence to zero, claimed in item (i).

In order to prove the item(ii), we need the following observation.

\begin{lemma}\label{bv}
If the function f is of bounded variation, then 
\begin{equation}
|R_j(f)| \leq \frac{TV(f)}{j},
\end{equation}
where $TV(f)$ denotes the total variation of $f$.

\begin{proof}
Since $f$ is of bounded variation, it can be written as a difference of two non-decreasing functions
\begin{equation}
f=f^+ - f^-.
\end{equation}

Using \eqref{Rj},
\begin{equation}
R_j(f^+)= \sum_{k=0}^{j-1} \int_{\frac kj}^{\frac{k+1}{j}}  ( f^+ (\frac kj)- f^+(x)) dx. 
\end{equation}
So that
\begin{equation}
|R_j(f^+)| \leq \frac{1}{j} \sum_{k=0}^{j-1} \left( f^+(\frac{k+1}{j}) - f^+(\frac kj) \right) \leq     \frac{1}{j} TV(f^+).
\end{equation}
Using the same argument for $f^-$ gives the result of Lemma~\ref{bv}.
\end{proof}
\end{lemma}

So, this shows that the sequence $(u_j(f))_{j\geq1}$ is bounded, when $f$ is of bounded variation.
Now, using the item (i) of this Lemma and item (iii) of Lemma~\ref{cesarofacts}, we see that the sequence $(|u_j(f)|)_{j\geq1}$ and thus, $u_j(f)$ converges in $(C,\theta)$ sense to zero, for any value of $\theta>0$.

The last item is trivial, since the sequence $u(f)= (u_j(f))_{j\geq1}$ is bounded, say by the constant C. Indeed, then the Cesaro means of the sequence $u(f)^2=(u_j(f)^2)_{j\geq1}$ are bounded, for any $\theta>0$ by
\begin{equation}
\sigma_N^{\theta}(u(f)^2) \leq C \sigma_N^{\theta}(|u(f)|).
\end{equation}
This implies the $(C,\theta)$ convergence of the sequence $u(f)^2$.
\end{proof}
\end{lemma}

Thus, Lemma~\ref{simplefacts} shows that Theorem~\ref{Thm1general} implies Theorem~\ref{Thm}.
Indeed, the general assumptions needed in Theorem~\ref{Thm1general} about the Cesaro convergence of the sequences $u(f)$ and $u(f)^2$ are satisfied by 
Lemma~\ref{simplefacts}.

\section{Proof of Theorem \ref{Thm2general} and Theorem \ref{Thm2}}

\subsection{A simple Gaussian convergence result for series of Poisson random variables} \ \newline

We give here a result of convergence in distribution for the random variables
\begin{equation}\label{H}H_N(f)=\sum_{j=1}^NW_ju_j(f)\end{equation} 
to a Gaussian law, once centered and normalized.This result is again elementary since it only uses the fact that the random variables $W_j$'s are independent and Poisson.

 Here, we assume as in Theorem \ref{Thm2general} that
\begin{equation}\label{Hyp2}
\sum_{j=1}^{\infty}jR_j(f)^2 = \infty
\end{equation}
and that 
\begin{equation}\label{assumption_max_bd}
\max_{1\leq j \leq N} |u_j|=\max_{1\leq j \leq N} |jR_j| = o(( \sum_{j=1}^{N}jR_j(f)^2)^{\frac12} ).
\end{equation}

Let us denote by
\begin{equation}
\eta_N^2= \Var(H_N) = \theta \sum_{j=1}^{N}jR_j(f)^2.
\end{equation}

\begin{lemma}\label{gaussian}
Under these assumptions, the distribution of
$$\frac{H_N-\mathbb{E}[H_N]}{\eta_N}$$ 
converges weakly to $\mathcal{N}(0,1)$ as $N\rightarrow\infty$.

\begin{proof}
 Write $\widetilde{H}_N$ for

\begin{equation}
\frac{H_N-\mathbb{E}[H_N]}{\eta_N}=\frac{H_N-\mathbb{E}[H_N]}{\sqrt{Var[H_N]}}=\sum_{j=1}^N\frac{u_j(W_j-(\theta/j))}{\eta_N},
\end{equation}

then
\begin{equation}
 \log\mathbb{E}\left[e^{it\widetilde{H}_N}\right]= \sum_{j=1}^N\frac\theta j(e^{itu_j/\eta_N}-itu_j/\eta_N-1),
\end{equation}
which gives the distribution of $\widetilde{H}_N$ in its L\'{e}vy-Khintchine representation $(a_N, \theta M_N,\sigma_N^2)$, with 

\begin{equation}a_N=\sum_{j=1}^N\left(\frac{\theta}{j(\eta_N^2+u_j^2)}\left(\frac{-u_j^3}{\eta_N}\right)\right),\end{equation}
\begin{equation}M_N=\sum_{j=1}^N\frac 1j\delta_{u_j/\eta_N}\end{equation} and $\sigma_N=0$.

 We continue this proof by applying the L\'{e}vy-Khintchine Convergence Theorem (\cite{varadhan}, p. 62):

Consider a bounded continuous function $f$ such that $f(x)=0$ for $|x|<\delta$, then
\begin{equation}
\int fdM_N=\sum_{j=1}^N\frac\theta j f\left(\frac{u_j}{\eta_N}\right)\mathbbm{1}_{\left|\frac{u_j}{\eta_N}\right|>\delta}.
\end{equation} 
Under the assumption \eqref{assumption_max_bd}, $\int fdM_N=0$ for $N$ large enough, so that

\begin{equation}\label{lkconv1}\lim_{N\rightarrow\infty}\int fdM_N=\int fdM=0.\end{equation}

Again, using the assumption \eqref{assumption_max_bd}, we have that for any $\ell>0$, 
\begin{equation}\int_{-\ell}^\ell x^2dM_N+\sigma_N^2=\sum_{j=1}^N\frac\theta j\frac{u_j^2}{\eta_N^2}\mathbbm{1}_{|u_j|<\ell\eta_N}.\end{equation}
So for $N$ large enough, $\int_{-\ell}^\ell x^2dM_N=1$ and so
\begin{equation}\label{lkconv2}\lim_{N\rightarrow\infty}\int_{-\ell}^\ell x^2dM_N=1.\end{equation}

Moreover, for every $N$ define 
\begin{equation}\label{eps}\epsilon_N= \frac{\max_{1\leq j\leq N}|u_j|}{\eta_N},\end{equation}

then we can bound $a_N$ above by 
\begin{equation}
|a_N|\leq\sum_{j=1}^N\left(\frac{\theta/j}{\eta_N^2+u_j^2}(\epsilon_Nu_j^2)\right)\leq\frac{\epsilon_N}{\eta_N^2}\sum_{j=1}^N\frac\theta ju_j^2=\epsilon_N.\end{equation} 
By the assumption \eqref{assumption_max_bd}, we see that \begin{equation}\label{lkconv3}\lim_{N\rightarrow\infty}a_N=0.\end{equation}

By \eqref{lkconv1}, \eqref{lkconv2}, \eqref{lkconv3} and using Theorem 3.21, p. 62 in \cite{varadhan}, we see that $\widetilde{H}_N$ converges in distribution to the infinitely divisible distribution with L\'{e}vy-Khintchine representation $(a, M, \sigma^2)=(0,0,1)$, i.e. to the standard normal Gaussian $\mathcal{N}(0,1)$.

\end{proof}
\end{lemma}

\subsection{Proof of Theorem \ref{Thm2general}} \ \newline

We proceed now to the proof of Theorem~\ref{Thm2general}, by using the Feller coupling bounds proved in Section 4. 
Again, we assume here, as in Theorem~\ref{Thm2general}, that
\begin{equation}\label{Hyp2}
\sum_{j=1}^{\infty}jR_j(f)^2 = \infty
\end{equation}
and that 
\begin{equation}
\max_{1\leq j \leq N} |jR_j| = o(( \sum_{j=1}^{N}jR_j(f)^2)^{\frac12} )
\end{equation}
which can be rewritten
\begin{equation} 
\max_{1\leq j \leq N} |u_j| = o(\eta_N )
\end{equation}

By Lemma \ref{L1boundwithmax}, we know that

\begin{equation}
\mathbb{E}(|G_N-H_N|) \leq C(\theta) \max_{1\leq j\leq N} |u_j| = o(\eta_N)
\end{equation}

Again, denote by $\widetilde{H}_N:= \frac{H_N - \E(H_N)}{\eta_N}$ and $\widetilde{G}_N:=\frac{ G_N - \E(G_N)}{\eta_N}$.
 Then obviously,
 \begin{equation}
 \E(|\widetilde{G}_N - \widetilde{H}_N)|) = o(1),
 \end{equation}
 which, together with the convergence result (Lemma~\ref{gaussian}) for $H_N$ proves that  $\widetilde{G}_N$ converges in distribution to a standard Gaussian law N(0,1).\\
  
 Moreover, this also proves that 
 \begin{equation}
 \E(G_N) = \E(H_N)) + o(\eta_N)= \theta \sum_{j=1}^N R_j +o(\eta_N).
 \end{equation}
 
And, by Lemma \ref{L2boundwithmax}, we also know that

\begin{equation}
\mathbb{E}((G_N-H_N)^2) \leq C(\theta) \max_{1\leq j\leq N} u_j^2 =o(\eta_N^2).
\end{equation}

 Then obviously,
 \begin{equation}
 \E((\widetilde{G}_N - \widetilde{H}_N)^2) \leq 2 ( \E((G_N-H_N)^2) + \E(G_N - H_N)^2) = o(\eta_N^2).
 \end{equation}
 So that
 \begin{equation}
 | \sqrt{\Var(G_N)}- \sqrt{\Var(H_N)}| = o(\eta_N)
 \end{equation}
 and thus, 
 \begin{equation}
 \Var(G_N) \sim \Var(H_N)= \eta_N^2.
 \end{equation}
 From these three results, we get that $$\frac{G_N - \E(G_N)}{\sqrt{\Var(G_N)}}$$ converges in distribution to a standard Gaussian $\mathcal{N}(0,1)$. But $I_{\sigma, N}(f) - \E(I_{\sigma, N}(f))$ has the same distribution as $G_N - \E(G_N)$ and thus, $$\frac{I_{\sigma, N} - \E(I_{\sigma, N})}{\sqrt{\Var(I_{\sigma, N})}}$$ converges also in distribution to a $\mathcal{N}(0,1)$ distribution. We thus have proved the first statement of Theorem \ref{Thm2general}.
 Moreover, we have that
  \begin{equation}
\E(I_N)=\int_0^1f(x)dx  + \E(G_N) =\int_0^1f(x)dx + \theta \sum_{j=1}^N R_j +o(\eta_N),
 \end{equation}
 which is the second statement of Theorem \ref{Thm2general}.
 Finally,
  \begin{equation}
 \Var(I_{\sigma, N})= \Var(G_N) \sim \Var(H_N)= \eta_N^2,
 \end{equation}
 which is the third statement. We have completed the proof of Theorem \ref{Thm2general}.
 
 \subsection{Proof of Theorem \ref{Thm2}}\ \newline
 
We prove here how Theorem~\ref{Thm2general} implies Theorem~\ref{Thm2}.
In Theorem \ref{Thm2} we assumed that $f$ is of bounded variation, which implies, as we have seen, that $u_j(f) = O(1)$, and thus, that
\begin{equation}
 \max_{1\leq j \leq N} |u_j| = o(\eta_N)
 \end{equation} 
 since the sequence $\eta_N$ is assumed to diverge. 
 This proves that the hypothesis of Theorem~\ref{Thm2} are satisfied under those of Theorem~\ref{Thm2general}.
 Thus we get that the conclusions of Theorem~\ref{Thm2general} are valid. They are almost exactly the same as the conclusions of Theorem \ref{Thm2}.
 The only thing left to prove is the item (ii). But  using Lemma \ref{L1boundwithmax},  \ref{meanH}, and the fact that $u_j(f) = jR_j(f)=O(1)$, we have that
 \begin{equation}
 \E[{I_{\sigma, N}}(f)]= N\int_0^1 f(x)dx+\sum_{j=1}^NR_j(f) + O(1).
 \end{equation}
The bound
\begin{equation}
\sum_{j=1}^NR_j(f) = O( \log N)
\end{equation}
is trivial since again $ R_j = O(\frac{1}{j})$.
With this we have derived Theorem~\ref{Thm2} from Theorem~\ref{Thm2general}.


\section{The expectation and the variance}
\label{expectation_and_variance}
For the sake of completeness, we give here the explicit expressions for the expectation and the variance of $I_{\sigma, N}$, when $\sigma$ is chosen from $\mathcal{S}_N$ by the Ewens distribution with parameter $\theta$. The basic computations for the expectation and the variance of the cycle counts can be simply derived by the following formula established by Watterson \cite{watterson} (see Arratia, Barbour and Tavar\'{e} \cite{arratiabarbourtavare}, (4.7), p. 68): \\
For every $b\geq 1$, $(r_1,\dots,r_b)\geq 0$, 
\begin{equation}\label{generalformula}
\mathbb{E}\left[\prod_{j=1}^b\alpha_j^{[r_j]}\right]=\mathbbm{1}_{m\leq N}\binom{N-m+\gamma}{N-m}\binom{N+\gamma}{N}^{-1}\prod_{j=1}^b\left(\frac \theta j\right)^{r_j},
\end{equation}
where $m=\sum_{j=1}^bjr_j$, $x^{[r]}=x(x-1)\dots(x-r+1)$ and $\gamma=\theta-1$. \\
Thus, the mean and the variance of $I_{\sigma, N}$ can be easily computed.

\begin{lemma}
\begin{eqnarray}\label{mean}
\mathbb{E}[I_{\sigma,N}(f)]&=&N\int_0^1f(x)dx+\sum_{j=1}^N\mathbb{E}[\alpha_j(\sigma)]u_j(f)\nonumber\\
&=&N\int_0^1f(x)dx+\theta\sum_{j=1}^N\Psi_N(j)R_j(f)\nonumber\\
\end{eqnarray}

\begin{proof}
From the general formula \eqref{generalformula}, we easily see that for any $j\geq 1$ and any $\theta>0$,
\begin{equation}
\mathbb{E}_\theta[\alpha_j]=\frac \theta j \Psi_N(j)\mathbbm{1}_{j\leq N}.
\end{equation}
Thus, \eqref{mean} follows immediately.
\end{proof}
\end{lemma}

The variance of $I_{\sigma,N}$ is given by the following lemma:
\begin{lemma}
\begin{eqnarray}\label{variance}
&&\Var[I_{\sigma,N}(f)]=\Var\left[\sum_{j=1}^N\alpha_ju_j(f)\right]\nonumber\\
&&\quad=\theta\sum_{j=1}jR_j^2\Psi_N(j)+\theta^2\sum_{j,j'\leq N}R_jR_{j'}\left(\Psi_N(j+j')\mathbbm{1}_{j+j'\leq N}-\Psi_N(j)\Psi_N(j')\right)\nonumber\\
&&\quad=\eta_N^2+\theta\sum_{j=1}jR_j^2(\Psi_N(j)-1)\nonumber\\
&&\quad\qquad+\theta^2\sum_{j,j'\leq N}R_jR_{j'}\left(\Psi_N(j+j')\mathbbm{1}_{j+j'\leq N}-\Psi_N(j)\Psi_N(j')\right)
\end{eqnarray}

\begin{proof}
Again, from the general formula \eqref{generalformula}, we easily see that for any $j\geq 1$ and any $\theta>0$,
\begin{equation}
\mathbb{E}_\theta[\alpha_j\alpha_{j'}]=\frac {\theta^2}{jj'} \Psi_N(j+j')\mathbbm{1}_{j+j'\leq N}
\end{equation}
and
\begin{equation}
\mathbb{E}_\theta[\alpha_j^2]=\frac {\theta^2}{j^2} \Psi_N(2j)\mathbbm{1}_{j\leq N/2}+\frac \theta j \Psi_N(j)\mathbbm{1}_{j\leq N}.
\end{equation}

The variance of $\alpha_j$ is therefore given by
\begin{equation}
\textrm{Var}_\theta[\alpha_j]=\frac \theta j \Psi_N(j)\mathbbm{1}_{j\leq N}+\frac {\theta^2}{ j^2} \Psi_N(2j)\mathbbm{1}_{j\leq N/2}-\frac {\theta^2}{ j^2} \Psi_N(j)^2\mathbbm{1}_{j\leq N}
\end{equation}

and the covariance by
\begin{equation}
\textrm{Cov}_\theta[\alpha_j,\alpha_{j'}]=\frac {\theta^2}{ jj'} \Psi_N(j+j')\mathbbm{1}_{j+j'\leq N}-\frac {\theta^2}{ jj'} \Psi_N(j)\Psi_N(j')\mathbbm{1}_{j\leq N}\mathbbm{1}_{j'\leq N},
\end{equation}
for $j\neq j'$.
Then, the variance of $I_{\sigma,N}$ given in \eqref{variance} follows immediately.
\end{proof}
\end{lemma}

\begin{remark}
The case where $\theta=1$ is particularly simple. Indeed then
\begin{equation}
\mathbb{E}[I_{\sigma,N}(f)]=N\int_0^1f(x)dx+\theta\sum_{j=1}^N R_j(f)\nonumber\\
\end{equation}
and
\begin{equation}
\Var [I_{\sigma,N}(f)]=\eta_N^2.\nonumber\\
\end{equation}
Thus the asymptotic formulae we give in this work are then exact.
For general values of $\theta>0$, it is possible to derive these asymptotic expressions directly form the explicit formulae given in this Section, without using the bounds on the Feller coupling, but this is not a trivial matter, in particular when $\theta<1$.

\end{remark}

{\sc Acknowledgements}: Kim Dang wishes to thank Ashkan Nikeghbali for his constant support.  This work was initiated at MSRI, during the program on Random Matrices, during the Fall 2010. Both authors thank the organizers for their invitation, and the MSRI  for its generous support. G\'erard Ben Arous is particularly grateful for the granting of an Eisenbud Professorship. Both authors acknowledge the generous support of New York University in AbuDhabi, where this work was completed. The work of G\'erard Ben Arous was supported in part by the National Science Foundation under grants DMS-0806180 and OISE-0730136.

\end{document}